\documentclass[twoside,reqno]{article}
\usepackage{amsmath, amssymb, amsthm, enumerate, fancyhdr, float}
\usepackage[mathcal]{euscript}
\newtheorem{theorem}{Theorem}[section]
\newtheorem{lemma}[theorem]{Lemma}
\newtheorem{proposition}[theorem]{Proposition}
\newtheorem{corollary}{Corollary}[theorem]
\newtheorem{definition}{Definition}[section]
\newtheorem{remark}{Remark}[section]
\newtheorem{notation}{Notation}[section]

\title{Topologies on sets of polynomial knots and the homotopy types of the respective spaces}
\author{Hitesh Raundal\\{\it\small Harish-Chandra Research Institute, Chhatnag Road, Jhunsi, Prayagraj 211019, India}\\{\it\small hiteshrndl@gmail.com}}
\date{}

\begin{document}
\maketitle

\begin{abstract}
{A polynomial knot in $\mathbb{R}^n$ is a smooth embedding of $\mathbb{R}$ in $\mathbb{R}^n$ such that the component functions are real polynomials. In the earlier paper with Mishra, we have studied the space $\mathcal{P}$ of polynomial knots in $\mathbb{R}^3$ with the inductive limit topology coming from the spaces $\mathcal{O}_d$ for $d\geq3$, where $\mathcal{O}_d$ is the space of polynomial knots in $\mathbb{R}^3$ with degree $d$ and having some conditions on the degrees of the component polynomials. In the same paper, we have proved that the space of polynomial knots in $\mathbb{R}^3$ has the same homotopy type as $S^2$. The homotopy type of the space is the mere consequence of the topology chosen. If we have another topology on $\mathcal{P}$, the homotopy type may change. With this in mind, we consider in general the set $\mathcal{L}^n$ of polynomial knots in $\mathbb{R}^n$ with various topologies on it and study the homotopy type of the respective spaces. Let $\mathcal{L}$ be the union of the sets $\mathcal{L}^n$ for $n\geq1$. We also explore the homotopy type of the space $\mathcal{L}$ with some natural topologies on it.\\[7pt]
}
\end{abstract}

\section{Introduction}\label{sec1}
 
The space of knots as embeddings of $S^1$ in $S^3$ with some specific behavior at the north pole can be approximated by the space of embeddings of $\mathbb{R}$ into $\mathbb{R}^n$ of the form
\begin{equation*}
t\mapsto\left(t^d+a_1^1t^{d-1}+\cdots+a_{d-1}^1t,\hskip0.2mm t^d+a_1^2t^{d-1}+\cdots+a_{d-1}^2t,\hskip0.2mm\cdots,\hskip0.2mm t^d+a_1^nt^{d-1}+\cdots+a_{d-1}^nt\right),
\end{equation*}
\noindent for example see \cite{bl, vav1,vav2}. Let $\mathcal{K}_d^n$ be the space of maps from $\mathbb{R}$ to $\mathbb{R}^n$ of the form above and let $\Sigma_d^n$ be the discriminant of this space. In \cite{vav2}, Vassiliev discussed and proved the following result:
\begin{theorem}\cite[Theorem 1]{vav2}
For $n\geq2$ the space $\mathcal{K}_3^n\setminus\Sigma_3^n$ is contractible and the space $\mathcal{K}_4^n\setminus\Sigma_4^n$ is homology equivalent to $S^{n-2}$. For $n\geq4$ the space $\mathcal{K}_4^n\setminus\Sigma_4^n$ is also homotopy equivalent to $S^{n-2}$. The generator of the group $H^{n-2}\left(\mathcal{K}_4^n\setminus\Sigma_4^n\right)$ is equal to the linking number in $\mathcal{K}_4^n$ with the set of maps above with singular points.
\end{theorem}
\noindent This result is about investigating the homology and homotopy of some specific kind of space of embeddings of $\mathbb{R}$ in $\mathbb{R}^n$ whose component functions are real polynomials. With this motivation, Durfee and O'Shea \cite{do} introduced the space $\mathcal{M}_d^n\setminus\Sigma_d^n$, where $\mathcal{M}_d^n$ is the space of maps from $\mathbb{R}$ to $\mathbb{R}^n$ whose component functions are real polynomials of degree at most $d$ (and one of them is exactly equal to $d$), and $\Sigma_d^n$ is the discriminant of the space $\mathcal{M}_d^n$. If two members of the space $\mathcal{M}_d^n\setminus\Sigma_d^n$ are path equivalent, then they are topologically equivalent (i.e. if two members belong to the same path component of the space $\mathcal{M}_d^n\setminus\Sigma_d^n$, then their extensions as embeddings of $S^1$ in $S^3$ are ambient isotopic), see \cite{do}. {\it A smooth embedding of $\mathbb{R}$ in $\mathbb{R}^n$ given by $t\mapsto\left(f_1(t),f_2(t),\ldots,f_n(t)\right)$, where $f_1,f_2,\ldots,f_n$ are real polynomials, is referred as polynomial knot.} Note that the spaces $\mathcal{K}_d^n\setminus\Sigma_d^n$ and $\mathcal{M}_d^n\setminus\Sigma_d^n$ are spaces of polynomial knots of degree $d$ with some specific conditions on the component polynomials.

Using the Weierstrass approximation theorem, Shatri \cite{ars} proved that every knot has polynomial representation. That is, for a given knot as an embedding of $S^1$ in $S^3$, there exists a polynomial knot $\phi:\mathbb{R}\hookrightarrow\mathbb{R}^n$ whose extension is ambient isotopic to the original one. Furthermore, it has been proved that if two polynomial knots represent the same knot-type, then there exists a polynomial isotopy $F:[0,1]\times\mathbb{R}\to\mathbb{R}^3$ which connects them, see \cite[Theorem 2.3]{rs}. If we have a long knot as an embedding of $\mathbb{R}$ in $\mathbb{R}^3$, then it can be rolled to the infinity to make it into a straight line. In other words, for a long knot $\phi$, there exists an isotopy $H:[0,1]\times\mathbb{R}\to\mathbb{R}^3$ of long knots that connects $\phi$ to an unknot (i.e. $H_s:\mathbb{R}\to\mathbb{R}^3$ given by $t\mapsto H(s,t)$ is a long knot for $s\in[0,1]$, $H_0=\phi$ and $H_1$ is an unknot). Since a polynomial knot is also a long knot, there exists an isotopy of long knots (i.e. an isotopy of the type $H$ above) connecting a given polynomial knot to an unknot. Here, it is natural to ask that is there a polynomial isotopy that connects a given polynomial knot to a linear one (to an unknot)? We have the affirmative answer to this question and this fact has been proved in my earlier paper \cite{rm} with Mishra. In the same paper, we have introduced the spaces $\mathcal{O}_d$, $\mathcal{P}_d$ and $\mathcal{Q}_d$ of polynomial knots as embeddings of $\mathbb{R}$ in $\mathbb{R}^3$ with the degrees of component polynomials less than or equal to $d$ and having some specific conditions on the degrees. We proved that if two polynomial knots belong to the same path component of the space $\mathcal{Q}_d$, then their extensions as embeddings of $S^1$ in $S^3$ are ambient isotopic (see \cite[Theorem 4.6]{rm}). In other words, if two knots as embeddings of $S^1$ in $S^3$ have polynomial representations that belong to different path components of the space $\mathcal{Q}_d$, then the knots are not ambient isotopic (are not equivalent as knots). In \cite{mr}, we have produced polynomial representations of all knots up to six crossings in their minimal degree except for the knots $5_2$, $6_1$, $6_2$ and $6_3$ which we have represented in degree $7$ (see Rolfsen's table in \cite[Appendix E]{vm} and \cite{dr} for the notations of knots). If one could prove that each path component of the space $\mathcal{Q}_6$ contains a polynomial knot representing one of the knot $0_1$, $3_1$ or $4_1$, then there will not be representations for the knots $5_2$, $6_1$, $6_2$ and $6_3$ in degree $6$ and hence the minimal polynomial degree of these knots will be $7$. With this regard, one can see that by studying spaces of polynomial knots would help in drawing some inferences in the point of view of knot theory. Reader may refer to \cite{cdm, vm, dr} for basic knot theory.

In \cite{rm}, we have explored the homotopy types of the spaces $\mathcal{O}_d$ for $d\geq3$ and proved that they have the same homotopy type as $S^2$ (see \cite[Theorem 4.13]{rm}). Using this result, we have proved that the space $\mathcal{P}$ of polynomial knots as embeddings of $\mathbb{R}$ in $\mathbb{R}^3$ with the inductive limit topology coming from the spaces $\mathcal{O}_d$ for $d\geq3$ has the same homotopy type as $S^2$ (see \cite[Corollary 4.15]{rm}). The space $\mathcal{P}$ can be given various topologies that come in a natural way. According to the topology given, the homotopy type of the space $\mathcal{P}$ may change. To look into this problem, we consider more generalized spaces $\mathcal{L}^n$, for $n\geq1$, of polynomial knots together with four types of various topologies on them. The space $\mathcal{L}^n$ is the space of polynomial knots as embeddings of $\mathbb{R}$ in $\mathbb{R}^n$. Note that the space $\mathcal{L}^3$ is the space $\mathcal{P}$ discussed in \cite{rm}. We also consider the space $\mathcal{L}$ as union of the spaces $\mathcal{L}^n$ for $n\geq1$. By identifying a polynomial knot $\phi:\mathbb{R}\hookrightarrow\mathbb{R}^n$ with the embedding $\tilde{\phi}:\mathbb{R}\hookrightarrow\mathbb{R}^\infty$ via the standard embedding of $\mathbb{R}^n$ in $\mathbb{R}^\infty$ and then eventually with a $\Lambda$-tuple $\left(\phi_{ij}\right)_{i,j}$, we can think of the sets $\mathcal{L}^n$ and $\mathcal{L}$ as subsets of $\mathbb{R}^\Lambda$, where $\Lambda=\mathbb{Z}^+\times\mathbb{N}$ and $\phi_{ij}$ is the coefficient of $t^j$ in the $i^{\text{th}}$ component of $\tilde{\phi}$. The sets $\mathcal{L}^n$ and $\mathcal{L}$ can be given the subspace topologies that inherit from the box and the product topologies of $\mathbb{R}^\Lambda$. We also have the topologies on $\mathcal{L}^n$ and $\mathcal{L}$ induced by the metrics $d_r$ (for $r\geq1$) and $d_\infty$ given by 
\begin{equation*}
d_r(\phi, \psi)=\left(\sum_{i,j}\left|\phi_{ij}-\psi_{ij}\right|^r\right)^{1/r}\qquad\text{and}\qquad d_\infty(\phi, \psi)=\sup_{i,j} \left| \phi_{ij}-\psi_{ij}\right|
\end{equation*}
\noindent for $\phi,\psi\in\mathcal{L}$. In this paper, we study the homotopy types of the spaces $\mathcal{L}^n$ and $\mathcal{L}$ with respect to the topologies discussed above. We see that the homotopy types of the spaces are independent on the topologies chosen and the results are compatible with the result proved in \cite{rm} regarding the homotopy type of the space $\mathcal{P}$.

This paper is organized as follows: In Section \ref{sec2}, we define sets $\mathcal{L}^n$ and $\mathcal{L}$, and discuss various topologies on them. In Section \ref{sec3}, we compare the topologies described in Section \ref{sec2}. In Section \ref{sec4}, we prove the results (see Theorem \ref{th1} and Theorem \ref{th2}) regarding the homotopy types of the spaces $\mathcal{L}^n$ and $\mathcal{L}$ with the topologies described in Section \ref{sec2}.

\section{Spaces $\mathcal{L}^n$ and $\mathcal{L}$}\label{sec2}

For a positive integer $n$, we define a polynomial knot in $\mathbb{R}^n$ as follows:

\begin{definition}
A polynomial knot in $\mathbb{R}^n$ is a smooth embedding $\mathbb{R}\hookrightarrow\mathbb{R}^n$ such that the component functions are real polynomials.
\end{definition}

In other words, a polynomial knot in $\mathbb{R}^n$ is a map $\phi:\mathbb{R}\to\mathbb{R}^n$ given by $t\mapsto\left(\phi_1(t),\phi_2(t),\ldots,\phi_n(t)\right)$ such that it satisfies the following conditions:
\begin{enumerate}[(1)]
\item $\phi_i\in\mathbb{R}[t]$ for $i=1,2,\ldots,n$, 
\item $\phi(s)\neq\phi(t)$ for $s\neq t$, and 
\item $\phi'(t)\neq0$ for $t\in\mathbb{R}$.
\end{enumerate}

\begin{notation}
For $n\geq1$, let $\mathcal{L}^n$ denote the set of all polynomial knots in $\mathbb{R}^n$, and let $\mathcal{L}=\bigcup_{n\in\mathbb{Z}^+}\mathcal{L}^n$.
\end{notation}

\begin{notation}
Let $\Lambda$ denote the set $\mathbb{Z}^+\times\mathbb{N}$, where $\mathbb{Z}^+$ is the set of all positive integers and $\mathbb{N}$ is the set of all nonnegative integers.
\end{notation}

Given a polynomial knot $\phi:\mathbb{R}\hookrightarrow\mathbb{R}^n$, it can be thought of as an embedding $\tilde{\phi}$ of $\mathbb{R}$ in $\mathbb{R}^\infty$ via the standard embedding of $\mathbb{R}^n$ in $\mathbb{R}^\infty$. Let us write $\tilde{\phi}(t)=\sum_j\left(\phi_{ij}\right)_i t^j$ for $t\in\mathbb{R}$, where $\phi_{ij}$ (for $(i,j)\in\Lambda$) is the coefficient of $t^{\hskip0.3mm j}$ in the $i^{\hskip0.3mm\text{th}}$ component of $\tilde{\phi}$. We can identify the polynomial knot $\phi$ with the $\Lambda$-tuple $\left(\phi_{ij}\right)_{i,j}$ of real numbers. With this identification of a polynomial knot, we can think of the sets $\mathcal{L}^n$ (for $n\geq1$) and $\mathcal{L}$ as subsets of $\mathbb{R}^\Lambda$, and thus they can be redefined as follows:
\begin{align*}
\mathcal{L}&=\left\{\left(\phi_{ij}\right)_{i,j}\in\mathbb{R}^\Lambda\hskip0.1mm\left|\right.\text{(i)}\;\phi_{ij}=0\;\text{for all but finitely many}\;(i,j)\in\Lambda,\text{and (ii) the map}\vphantom{t\mapsto\sum_j\left(\phi_{ij}\right)_i t^j}\right.\\
&\left.\hskip30mm\phi:\mathbb{R}\to\mathbb{R}^\infty\;\text{given by}\;t\mapsto\sum_j\left(\phi_{ij}\right)_i t^j\;\text{is a smooth embedding}\,\right\},\\[2pt]
\mathcal{L}^n&=\left\{\left(\phi_{ij}\right)_{i,j}\in\mathcal{L}\hskip0.3mm\left|\right.\phi_{ij}=0\;\text{if}\;i>n\;\text{and}\;j\geq0\right\}.
\end{align*}

We have the subspace topologies on $\mathcal{L}^n$ and $\mathcal{L}$ that inherit from the box and the product topologies of $\mathbb{R}^\Lambda$. The sets $\mathcal{L}^n$ and $\mathcal{L}$ can be equipped with the metric topologies. Let $r\geq1$ be any real number, and let $d_r$ be the metric on $\mathcal{L}$ defined as
\begin{equation}
d_r(\phi, \psi)=\left(\sum_{i,j}\left|\phi_{ij}-\psi_{ij}\right|^r\right)^{1/r}
\end{equation}
\noindent for $\phi,\psi\in\mathcal{L}$, where $\phi_{ij}$ and $\psi_{ij}$ (for $(i,j)\in\Lambda$) are the coefficients of $t^j$ in the $i^{\text{th}}$ components of $\phi$ and $\psi$ respectively. Denote an open ball centered at $\phi\in\mathcal{L}$ and of radius $\delta>0$ by $B_r(\phi,\delta)$, i.e. $B_r(\phi,\delta)=\left\{\psi\in\mathcal{L}:d_r(\phi,\psi)<\delta\right\}$. This type of balls generate a topology on $\mathcal{L}$. We have another metric $d_\infty$ on $\mathcal{L}$ defined as

\begin{equation}
d_\infty(\phi, \psi)=\sup_{i,j} \left| \phi_{ij}-\psi_{ij}\right|
\end{equation}
\noindent for $\phi,\psi\in\mathcal{L}$. This metric induces a topology on $\mathcal{L}$, i.e. the topology is generated by open balls of the type $B_\infty(\phi,\delta)=\left\{\psi\in\mathcal{L}:d_\infty(\phi,\psi)<\delta\right\}$, where $\phi\in\mathcal{L}$ and $\delta>0$. Since $\mathcal{L}^n$ is a subset of $\mathcal{L}$, we have the topologies on $\mathcal{L}^n$ induced by the metrics $d_r$ and $d_\infty$ when restricted to $\mathcal{L}^n$. Note that these topologies are generated by balls of the types $B_r^n(\phi,\delta)=\left\{\psi\in\mathcal{L}^n:d_r(\phi,\psi)<\delta\right\}$ and $B_\infty^n(\phi,\delta)=\left\{\psi\in\mathcal{L}^n:d_\infty(\phi,\psi)<\delta\right\}$ respectively, where $\phi\in\mathcal{L}^n$ and $\delta>0$. We denote the topologies on $\mathcal{L}^n$ and $\mathcal{L}$ as in Table \ref{tbl1}.

\begin{table}[H]
	\caption{Topologies on $\mathcal{L}^n$ and $\mathcal{L}$}\label{tbl1}
	\centering
	\begin{tabular}{|p{61mm}| p{26mm} | p{26mm}|}\hline
		Description & Notations for the topologies on $\mathcal{L}$ & Notations for the topologies on $\mathcal{L}^n$\\ \hline
		The subspace topology that inherit from the box topology of $\mathbb{R}^\Lambda$ & $\mathcal{T}_b$ & $\mathcal{T}_b^n$\\ \hline
		The subspace topology that inherit from the product topology of $\mathbb{R}^\Lambda$ & $\mathcal{T}_p$ & $\mathcal{T}_p^n$\\ \hline
		The metric topology induced by the metric $d_r$ on $\mathcal{L}$ (respectively by $d_r|_{\mathcal{L}^n}$) & $\mathcal{T}_r$ & $\mathcal{T}_r^n$\\ \hline
		The metric topology induced by the metric $d_\infty$ on $\mathcal{L}$ (respectively by $d_\infty|_{\mathcal{L}^n}$) & $\mathcal{T}_\infty$ & $\mathcal{T}_\infty^n$\\ \hline
	\end{tabular}
\end{table}

\begin{remark}
Suppose the spaces $\mathbb{R}^\Lambda$ and $\mathbb{R}^\infty$ have the box topologies or both have the product topologies. Then any bijection $\alpha:\Lambda\to\mathbb{Z}^+$ induces a homeomorphism $\Phi:\mathbb{R}^\Lambda\to\mathbb{R}^\infty$ given by $\left(\phi_{ij}\right)_{i,j}\mapsto\left(x_k\right)_k$, where $x_k=\phi_{ij}$ if $k=\alpha(i,j)$. With this type of homeomorphism in a mind, the spaces $\left(\mathcal{L}^n,\mathcal{T}_b^n\right)$ and $\left(\mathcal{L},\mathcal{T}_b\right)$ can be considered as subspaces of the space $\mathbb{R}^\infty$ with the box topology, and the spaces $\left(\mathcal{L}^n,\mathcal{T}_p^n\right)$ and $\left(\mathcal{L},\mathcal{T}_p\right)$ can be regarded as subspaces of the space $\mathbb{R}^\infty$ with the product topology.
\end{remark}

\section{Comparison between the topologies on $\mathcal{L}^n$ and $\mathcal{L}$}\label{sec3}

In this section, we compare the topologies $\mathcal{T}_b$, $\mathcal{T}_p$, $\mathcal{T}_r$, $\mathcal{T}_s$ (for $r>s\geq1$) and $\mathcal{T}_\infty$ on $\mathcal{L}$ and the topologies $\mathcal{T}_b^n$, $\mathcal{T}_p^n$, $\mathcal{T}_r^n$, $\mathcal{T}_s^n$ and $\mathcal{T}_\infty^n$ on $\mathcal{L}^n$.
 
\begin{lemma}\label{lem1}
Let $r$ and $s$ be real numbers such that $r\geq s\geq1$ and suppose $a_1,a_2,\ldots,a_k$ be nonnegative real numbers. Then $\left(\sum_{i=1}^k{a_i}^r\right)^{1/r}\leq\left(\sum_{i=1}^k{a_i}^s\right)^{1/s}$.
\end{lemma}

\begin{proof}
Let $u=\frac{r}{s}$ and let $l$ be the integer such that $u=l+v$ for some $0\leq v<1$. Note that $l\geq1$. Let $b_i={a_i}^s$ for $i=1,2,\ldots,k$ and $b=\sum_{i=1}^kb_i$. We have following:
\begin{equation}
\sum_{i=1}^k{b_i}^u\leq\sum_{i=1}^k{b_i}^v{b_i}^l\leq b^v\sum_{i=1}^k{b_i}^l\leq\left(\sum_{i=1}^kb_i\right)^v\left(\sum_{i=1}^kb_i\right)^l\leq\left(\sum_{i=1}^kb_i\right)^u
\end{equation}
\noindent Thus, $\sum_{i=1}^k{a_i}^r\leq\left(\sum_{i=1}^k{a_i}^s\right)^{r/s}$ and hence the result.
\end{proof}

\begin{proposition}\label{prop1}
For $r>s\geq1$, we have the relation $\mathcal{T}_p\subseteq\mathcal{T}_\infty\subseteq\mathcal{T}_r\subseteq\mathcal{T}_s\subseteq\mathcal{T}_b$ between the topologies on $\mathcal{L}$.
\end{proposition}

\begin{proof}
We prove the proposition in the following parts:

(1) To prove $\mathcal{T}_p\subseteq\mathcal{T}_\infty$: The topology $\mathcal{T}_p$ has the basis containing the sets of the type $\mathcal{L}\bigcap\prod_{i,j} U_{ij}$, where $U_{ij}$ is open in $\mathbb{R}$ for $(i,j)\in\Lambda$ and $U_{ij}=\mathbb{R}$ for all but finitely many $(i,j)\in\Lambda$. Let $U=\mathcal{L}\bigcap\prod_{i,j} U_{ij}$ be an open set in this basis. We show that $U$ is open in $(\mathcal{L},\mathcal{T}_\infty)$. Suppose an element $\phi\in U$ be given. It is sufficient to show that $B_\infty(\phi,\delta)\subseteq U$ for some $\delta>0$. Let $\Lambda_U=\left\{(i,j)\in\Lambda:U_{ij}\neq\mathbb{R}\right\}$. Note that $\Lambda_U$ is a finite set. Since $\phi_{ij}\in U_{ij}$ for $(i,j)\in\Lambda_U$, we can choose $\delta>0$ such that $\left(\phi_{ij}-\delta,\phi_{ij}+\delta\right)\subseteq U_{ij}$ for all $(i,j)\in\Lambda_U$. Now consider an open ball $B_\infty(\phi,\delta)$. Suppose $\psi\in B_\infty(\phi,\delta)$. Then $d_\infty(\phi,\psi)<\delta$ and hence $\psi_{ij}\in\left(\phi_{ij}-\delta,\phi_{ij}+\delta\right)$ for $(i,j)\in\Lambda_U$. This implies that $\psi_{ij}\in U_{ij}$ for $(i,j)\in\Lambda_U$. Note that $\psi_{ij}\in U_{ij}$ for $(i,j)\in\Lambda\setminus\Lambda_U$, since $U_{ij}=\mathbb{R}$ for $(i,j)\in\Lambda\setminus\Lambda_U$. Thus, $\psi\in U$. This shows that $B_\infty(\phi,\delta)\subseteq U$.

(2) To prove $\mathcal{T}_\infty\subseteq\mathcal{T}_r$: Let $B_\infty(\phi,\epsilon)$, for $\phi\in\mathcal{L}$ and $\epsilon>0$, be an open ball in $(\mathcal{L},d_\infty)$. We show that $B_\infty(\phi,\epsilon)$ is open in $(\mathcal{L},d_r)$. Suppose an element $\psi\in B_\infty(\phi,\epsilon)$ be given. It is enough to show that $B_r(\psi,\delta)\subseteq B_\infty(\phi,\epsilon)$ for some $\delta>0$. Take $\delta=\frac{\epsilon-d_\infty(\phi,\psi)}{2}$. Since $d_\infty(\psi,\omega)\leq d_r(\psi,\omega)$ for $\omega\in\mathcal{L}$, we have $d_\infty(\phi,\omega)\leq d_\infty(\phi,\psi)+d_\infty(\psi,\omega)\leq d_\infty(\phi,\psi)+d_r(\psi,\omega)\leq d_\infty(\phi,\psi)+\delta<\epsilon$ whenever $d_r(\psi,\omega)<\delta$. This shows that $B_r(\psi,\delta)\subseteq B_\infty(\phi,\epsilon)$.
 
(3) To prove $\mathcal{T}_r\subseteq\mathcal{T}_s$: By Lemma \ref{lem1}, we have $d_r(\phi,\psi)\leq d_s(\phi,\psi)$ for $\phi,\psi\in\mathcal{L}$. The rest of the proof follows in a similar way as in the previous part.

(4) To prove $\mathcal{T}_s\subseteq\mathcal{T}_b$: Let $B_s(\phi,\epsilon)$, for $\phi\in\mathcal{L}$ and $\epsilon>0$, be an open ball in $(\mathcal{L},d_s)$. We show that $B_s(\phi,\epsilon)$ is open in $(\mathcal{L},\mathcal{T}_b)$. Suppose an element $\psi\in B_s(\phi,\epsilon)$ be given. The topology $\mathcal{T}_b$ has the basis containing the sets of the type $\mathcal{L}\bigcap\prod_{i,j} V_{ij}$, where $V_{ij}$ is open in $\mathbb{R}$ for $(i,j)\in\Lambda$. It is sufficient to show that there is a basic open neighborhood $V$ of $\psi$ such that $V\subseteq B_s(\phi,\epsilon)$. Let $\delta=\min\left\{1/2,\epsilon-d_s(\phi,\psi)\right\}$ and $V=\mathcal{L}\bigcap\prod_{i,j} V_{ij}$, where $V_{ij}=\left(\psi_{ij}-\delta^{4i(j+1)},\hskip0.2mm\psi_{ij}+\delta^{4i(j+1)}\right)$ for $(i,j)\in\Lambda$. Note that $\psi\in V$. Suppose $\omega\in V$. Then $\left|\psi_{ij}-\omega_{ij}\right|<\delta^{4i(j+1)}$ for $(i,j)\in\Lambda$. Let $\Lambda_\omega=\left\{(i,j)\in\Lambda:\psi_{ij}\neq0\;\,\text{or}\;\,\omega_{ij}\neq0\right\}$. Note that $\Lambda_\omega$ is a finite set. We have $d_1(\psi,\omega)\leq\sum_{(i,j)\in\Lambda_\omega}\left|\psi_{ij}-\omega_{ij}\right|\leq\sum_{(i,j)\in\Lambda_\omega}\delta^{4i(j+1)}\leq\sum_{i=1}^\infty\sum_{j=0}^\infty\delta^{4i(j+1)}\leq\sum_{i=1}^\infty\frac{\delta^{4i}}{1-\delta^{4i}}\leq\sum_{i=1}^\infty\delta^{3i}\leq\frac{\delta^3}{1-\delta^3}\leq\delta^2$ (since $\delta\leq1/2$). By Lemma \ref{lem1}, $d_s(\psi,\omega)\leq d_1(\psi,\omega)$; thus, $d_s(\phi,\omega)\leq d_s(\phi,\psi)+d_s(\psi,\omega)\leq d_s(\phi,\psi)+d_1(\psi,\omega)\leq d_s(\phi,\psi)+\delta^2<\epsilon$ (since $\delta\leq\epsilon-d_s(\phi,\psi)$ and $\delta\leq1/2$). In other words $\omega\in B_s(\phi,\epsilon)$. This shows that $V\subseteq B_s(\phi,\epsilon)$.
\end{proof}

For $n\in\mathbb{Z}^+$ and for real numbers $r>s\geq1$, the topologies $\mathcal{T}_p^n$, $\mathcal{T}_\infty^n$, $\mathcal{T}_r^n$, $\mathcal{T}_s^n$ and $\mathcal{T}_b^n$ can be seen as the subspace topologies that inherit from the topologies $\mathcal{T}_p$, $\mathcal{T}_\infty$, $\mathcal{T}_r$, $\mathcal{T}_s$ and $\mathcal{T}_b$ respectively. Thus, the next corollary follows trivially from Proposition \ref{prop1}.

\begin{corollary}\label{cor1}
For $n\in\mathbb{Z}^+$ and for $r>s\geq1$, we have the relation $\mathcal{T}_p^n\subseteq\mathcal{T}_\infty^n\subseteq\mathcal{T}_r^n\subseteq\mathcal{T}_s^n\subseteq\mathcal{T}_b^n$ between the topologies on $\mathcal{L}^n$.
\end{corollary}

\begin{proposition}
For $n\in\mathbb{Z}^+$ and for $r>s\geq1$, we have the relation $\mathcal{T}_p\subsetneq\mathcal{T}_\infty\subsetneq\mathcal{T}_r\subsetneq\mathcal{T}_s\subsetneq\mathcal{T}_b$ between the topologies on $\mathcal{L}$ and the relation $\mathcal{T}_p^n\subsetneq\mathcal{T}_\infty^n\subsetneq\mathcal{T}_r^n\subsetneq\mathcal{T}_s^n\subsetneq\mathcal{T}_b^n$ between the topologies on $\mathcal{L}^n$.
\end{proposition}

\begin{proof}
Since the topologies $\mathcal{T}_p^n$, $\mathcal{T}_\infty^n$, $\mathcal{T}_r^n$, $\mathcal{T}_s^n$ and $\mathcal{T}_b^n$ can be viewed as the subspace topologies that inherit from the topologies $\mathcal{T}_p$, $\mathcal{T}_\infty$, $\mathcal{T}_r$, $\mathcal{T}_s$ and $\mathcal{T}_b$ respectively, it is enough to show that the set inclusions in Corollary \ref{cor1} are proper. We do this in the following parts:

(1) To show that $\mathcal{T}_\infty^n$ is strictly finer than $\mathcal{T}_p^n$: Let $\phi:\mathbb{R}\to\mathbb{R}^n$ be defined as $t\mapsto\left(t,0,\ldots,0\right)$. Note that $\phi\in\mathcal{L}^n$ and $B_\infty^n(\phi,1/2)\in\mathcal{T}_\infty^n$. We show that $B_\infty^n(\phi,1/2)\notin\mathcal{T}_p^n$. Suppose contrary that $B_\infty^n(\phi,1/2)\in\mathcal{T}_p^n$. Then there is a basic open neighborhood $U=\mathcal{L}^n\bigcap\prod_{i,j} U_{ij}$ of $\phi$ such that $U\subseteq B_\infty^n(\phi,1/2)$, where $U_{ij}$ is open in $\mathbb{R}$ for $(i,j)\in\Lambda$ and $U_{ij}=\mathbb{R}$ for all but finitely many $(i,j)\in\Lambda$. Choose an odd positive integer $k$ such that $U_{1k}=\mathbb{R}$. Let $\psi:\mathbb{R}\to\mathbb{R}^n$ be defined as $t\mapsto\left(t^k+t,0,\ldots,0\right)$. One can see that $\psi\in\mathcal{L}^n$, $\psi_{ij}\in U_{ij}$ for $(i,j)\in\Lambda\setminus\{(1,k)\}$ (since $\psi_{ij}=\phi_{ij}$ and $\phi_{ij}\in U_{ij}$), $\psi_{1k}\in U_{1k}$ (since $U_{1k}=\mathbb{R}$) and $d_\infty(\phi,\psi)=1$. Thus, $\psi\in U$ and $\psi\notin B_\infty^n(\phi,1/2)$. This is a contradiction.

(2) To show that $\mathcal{T}_r^n$ is strictly finer than $\mathcal{T}_\infty^n$: Let $\phi$ be defined as in the first part. Note that $B_r^n(\phi,1/2)\in\mathcal{T}_r^n$. We show that $B_r^n(\phi,1/2)\notin\mathcal{T}_\infty^n$. Suppose contrary that $B_r^n(\phi,1/2)\in\mathcal{T}_\infty^n$. Then there is a $\delta>0$ such that $B_\infty^n(\phi,\delta)\subseteq B_r^n(\phi,1/2)$. Choose an integer $k>\delta^{-r}$. Let $\omega:\mathbb{R}\to\mathbb{R}^n$ be defined as 
\begin{equation}
t\mapsto\left(k^{-\frac{1}{r}}\left(t^{2k+1}+t^{2k-1}+\cdots+t^3\right)+t,\,0,\,\ldots,\,0\right)
\end{equation}
\noindent We can see that $\omega\in\mathcal{L}^n$, $d_\infty(\phi,\omega)=k^{-\frac{1}{r}}<\delta$ and $d_r(\phi,\omega)=1$. Thus, $\omega\in B_\infty^n(\phi,\delta)$ and $\omega\notin B_r^n(\phi,1/2)$. This is a contradiction.

(3) To show that $\mathcal{T}_s^n$ is strictly finer than $\mathcal{T}_r^n$: Let $\phi$ be defined as in the first part. Note that $B_s^n(\phi,1/2)\in\mathcal{T}_s^n$. We show that $B_s^n(\phi,1/2)\notin\mathcal{T}_r^n$. Suppose contrary that $B_s^n(\phi,1/2)\in\mathcal{T}_r^n$. Then there is a $\delta>0$ such that $B_r^n(\phi,\delta)\subseteq B_s^n(\phi,1/2)$. Choose an integer $k>\delta^{\frac{rs}{s-r}}$. Let $\sigma:\mathbb{R}\to\mathbb{R}^n$ be defined as 
\begin{equation}
t\mapsto\left(k^{-\frac{1}{s}}\left(t^{2k+1}+t^{2k-1}+\cdots+t^3\right)+t,\,0,\,\ldots,\,0\right)
\end{equation}
\noindent One can see that $\sigma\in\mathcal{L}^n$, $d_r(\phi,\sigma)=k^{\frac{s-r}{rs}}<\delta$ and $d_s(\phi,\sigma)=1$. Thus, $\sigma\in B_r^n(\phi,\delta)$ and $\sigma\notin B_s^n(\phi,1/2)$. This is a contradiction.

(4) To show that $\mathcal{T}_b^n$ is strictly finer than $\mathcal{T}_s^n$: Let $\phi$ be defined as in the first part. Let $U=\mathcal{L}^n\bigcap\prod_{i,j} U_{ij}$, where $U_{ij}=\left(-1,\frac{3}{i(j+1)}\right)$. Note that $U\in\mathcal{T}_b^n$. Since $\phi_{ij}=0\in U_{ij}$ for $(i,j)\in\Lambda\setminus\{(1,1)\}$ and $\phi_{11}=1\in U_{11}$, $\phi\in U$. We show that $U\notin\mathcal{T}_s^n$. Suppose contrary that $U\in\mathcal{T}_s^n$. Then there is a $\delta>0$ such that $B_s^n(\phi,\delta)\subseteq U$. Choose an odd positive integer $k>\frac{4-\delta}{\delta}$. Let $\xi:\mathbb{R}\to\mathbb{R}^n$ be defined as $t\mapsto\left(\frac{4}{k+1}t^k+t,\,0,\,\ldots,\,0\right)$. Since $d_s(\phi,\xi)=\frac{4}{k+1}<\delta$ and $\xi_{1k}\geq\frac{4}{k+1}>\frac{3}{k+1}$ (i.e. $\xi_{1k}\notin U_{1k}$), we have $\xi\in B_s^n(\phi,\delta)$ and $\xi\notin U$. This is a contradiction.
\end{proof}

\section{Homotopy types of the spaces $\mathcal{L}^n$ and $\mathcal{L}$}\label{sec4}

In this section, we prove the following theorems regarding the homotopy types of the spaces $\mathcal{L}^n$ and $\mathcal{L}$ with the topologies as in Table \ref{tbl1}.

\begin{theorem}\label{th1}
The space $\mathcal{L}^n$ having any of the topologies $\mathcal{T}_b^n$, $\mathcal{T}_p^n$, $\mathcal{T}_r^n$ (for $r\geq1$) and $\mathcal{T}_\infty^n$ has the same homotopy type as $S^{n-1}$.
\end{theorem}

\begin{theorem}\label{th2}
The space $\mathcal{L}$ having any of the topologies $\mathcal{T}_b$, $\mathcal{T}_p$, $\mathcal{T}_r$ (for $r\geq1$) and $\mathcal{T}_\infty$ is contractible.
\end{theorem}

Let us go through some lemmas that will be used to prove Theorem \ref{th1} and Theorem \ref{th2}.

\begin{lemma}\label{lem2}
Let $F:\mathbb{R}^\Lambda\to\mathbb{R}^\infty$ be defined as $\left(\phi_{ij}\right)_{i,j}\mapsto\left(x_i\right)_i$, where $x_i=\phi_{i1}$ for $i\in\mathbb{Z}^+$. Suppose $G:\mathbb{R}^\infty\to\mathbb{R}^\Lambda$ be defined as $\left(x_i\right)_i\mapsto\left(\phi_{ij}\right)_{i,j}$, where $\phi_{ij}=0^{\left|j-1\right|}x_i$ for $\left(i,j\right)\in\Lambda$ and $0^0=1$. Suppose $\mathbb{R}^\Lambda$ and $\mathbb{R}^\infty$ have the box topologies or both have the product topologies. Then the maps $F$ and $G$ are continuous.
\end{lemma}

\begin{proof}
(1) To prove $F$ is continuous: Let $U=\prod_i U_i$ be a basic open set in the box topology of $\mathbb{R}^\infty$ (i.e. $U_i$ is open in $\mathbb{R}$ for $i\in\mathbb{Z}^+$). One can see that $F^{-1}(U)=\prod_{i,j} U_{ij}$, where $U_{i1}=U_i$ for $i\in\mathbb{Z}^+$ and $U_{ij}=\mathbb{R}$ for $(i,j)\in\mathbb{Z}^+\times\mathbb{N}\setminus\{1\}$. Thus, $F^{-1}(U)$ is open in the box topology of $\mathbb{R}^\Lambda$. If $U_i=\mathbb{R}$ for all but finitely many $i\in\mathbb{Z}^+$, i.e. if $U$ is a basic open set in the product topology of $\mathbb{R}^\infty$, then $F^{-1}(U)$ is open in the product topology of $\mathbb{R}^\Lambda$.

(2) To prove $G$ is continuous: Let $V=\prod_{i,j} V_{ij}$ be a basic open set in the box topology of $\mathbb{R}^\Lambda$ (i.e. $V_{ij}$ is open in $\mathbb{R}$ for $(i,j)\in\Lambda$). If there is an $(i,j)\in\mathbb{Z}^+\times\mathbb{N}\setminus\{1\}$ such that $0\notin V_{ij}$, then $G^{-1}(V)$ is an empty set. If $0\in V_{ij}$ for all $(i,j)\in\mathbb{Z}^+\times\mathbb{N}\setminus\{1\}$, then $G^{-1}(V)=\prod_i V_i$, where $V_i=V_{i1}$ for $i\in\mathbb{Z}^+$. Thus, in either case, the set $G^{-1}(V)$ is open in the box topology of $\mathbb{R}^\infty$. If $V_{ij}=\mathbb{R}$ for all but finitely many $(i,j)\in\Lambda$, i.e. if $V$ is a basic open set in the product topology of $\mathbb{R}^\Lambda$, then $G^{-1}(V)$ is open in the product topology of $\mathbb{R}^\infty$.
\end{proof}

\begin{notation}
Let $\mathcal{E}=\left\{\left(x_i\right)_i\in\mathbb{R}^\infty\setminus\{0\}\hskip0.2mm\left|\right.x_i=0\;\text{for all but finitely many}\:\,i\in\mathbb{Z}^+\right\}$, and for $n\geq1$, let $\mathcal{E}^n=\left\{\left(x_i\right)_i\in\mathbb{R}^\infty\setminus\{0\}\hskip0.2mm\left|\right.x_i=0\:\,\text{for all}\:\,i>n\right\}$.
\end{notation}

We have the subspace topologies on $\mathcal{E}$ and $\mathcal{E}^n$ that inherit from the box and the product topologies of $\mathbb{R}^\infty$. Let $r\geq1$ be a real number and let $\rho_r$ be a metric on $\mathcal{E}$ defined as
\begin{equation}
\rho_r(x, y)=\left(\sum_i\left|x_i-y_i\right|^r\right)^{1/r}
\end{equation}
\noindent for $x,y\in\mathcal{E}$. This metric induces a topology on $\mathcal{E}$, i.e. the topology is generated by open balls of the type $C_r(x,\delta)=\left\{y\in\mathcal{E}:\rho_r(x,y)<\delta\right\}$ for $x\in\mathcal{E}$ and $\delta>0$. We have another metric $\rho_\infty$ on $\mathcal{E}$ defined as
\begin{equation}
\rho_\infty(x,y)=\sup_i\left|x_i-y_i\right|
\end{equation}
\noindent for $x,y\in\mathcal{E}$. Let $C_\infty(x,\delta)$ be an open ball in the metric space $(\mathcal{E},\rho_\infty)$ centered at $x\in\mathcal{E}$ and of radius $\delta>0$. This type of balls generate a topology on $\mathcal{E}$. Since $\mathcal{E}^n$ (for $n\geq1$) is a subset of $\mathcal{E}$, the set $\mathcal{E}^n$ has the topologies induced by the metrics $\rho_r$ and $\rho_\infty$ when restricted to it. We denote the topologies on $\mathcal{E}$ and $\mathcal{E}^n$ as in Table \ref{tbl2}.

\begin{table}[H]
\caption{Topologies on $\mathcal{E}$ and $\mathcal{E}^n$}\label{tbl2}
\centering
\begin{tabular}{|p{61mm}| p{26mm} | p{26mm}|}\hline
Description & Notations for the topologies on $\mathcal{E}$ & Notations for the topologies on $\mathcal{E}^n$\\ \hline
The subspace topology that inherit from the box topology of $\mathbb{R}^\infty$ & $\mathcal{S}_b$ & $\mathcal{S}_b^n$\\ \hline
The subspace topology that inherit from the product topology of $\mathbb{R}^\infty$ & $\mathcal{S}_p$ & $\mathcal{S}_p^n$\\ \hline
The metric topology induced by the metric $\rho_r$ on $\mathcal{E}$ (respectively by $\rho_r|_{\mathcal{E}^n}$) & $\mathcal{S}_r$ & $\mathcal{S}_r^n$\\ \hline
The metric topology induced by the metric $\rho_\infty$ on $\mathcal{E}$ (respectively by $\rho_\infty|_{\mathcal{E}^n}$) & $\mathcal{S}_\infty$ & $\mathcal{S}_\infty^n$\\ \hline
\end{tabular}
\end{table}

\begin{remark}
One can check that $F(\mathcal{L}^n)\subseteq\mathcal{E}^n$, $F(\mathcal{L})\subseteq\mathcal{E}$, $G(\mathcal{E}^n)\subseteq\mathcal{L}^n$ and $G(\mathcal{E})\subseteq\mathcal{L}$, where $F$ and $G$ be the maps defined in Lemma \ref{lem2}.
\end{remark}

\begin{lemma}\label{lem3}
Let $f:\mathcal{L}\to\mathcal{E}$ and $g:\mathcal{E}\to\mathcal{L}$ be the restrictions of the maps $F$ and $G$ respectively (where $F$ and $G$ be the maps defined in Lemma \ref{lem2}). Suppose $\mathcal{E}$ and $\mathcal{L}$ have the topologies
\begin{enumerate}[(1)]
\item $\mathcal{S}_b$ and $\mathcal{T}_b$ respectively, or
\item $\mathcal{S}_p$ and $\mathcal{T}_p$ respectively, or
\item $\mathcal{S}_r$ and $\mathcal{T}_r$ (for $r\geq1$) respectively, or
\item $\mathcal{S}_\infty$ and $\mathcal{T}_\infty$ respectively.
\end{enumerate}
Then $f$ and $g$ are continuous.
\end{lemma}

For cases (1) and (2), the proof of Lemma \ref{lem3} is straight forward from Lemma \ref{lem2}. For cases (3) and (4), it is easy to check that $f$ and $g$ are Lipschitz continuous.

With the same type of topologies (i.e. the topologies having the same subscripts as in Table \ref{tbl1} and Table \ref{tbl2}), the spaces $\mathcal{L}^n$ and $\mathcal{E}^n$ are subspaces of the spaces $\mathcal{L}$ and $\mathcal{E}$ respectively. The next lemma follows immediately from Lemma \ref{lem3}.

\begin{lemma}\label{lem4}
Let $f_n:\mathcal{L}^n\to\mathcal{E}^n$ and $g_n:\mathcal{E}^n\to\mathcal{L}^n$ be the restrictions of the maps $F$ and $G$ respectively (where $F$ and $G$ be the maps defined in Lemma \ref{lem2}). Suppose $\mathcal{E}^n$ and $\mathcal{L}^n$ have the topologies
\begin{enumerate}[(1)]
\item $\mathcal{S}_b^n$ and $\mathcal{T}_b^n$ respectively, or
\item $\mathcal{S}_p^n$ and $\mathcal{T}_p^n$ respectively, or
\item $\mathcal{S}_r^n$ and $\mathcal{T}_r^n$ (for $r\geq1$) respectively, or
\item $\mathcal{S}_\infty^n$ and $\mathcal{T}_\infty^n$ respectively.
\end{enumerate}
Then $f_n$ and $g_n$ are continuous.
\end{lemma}

\begin{lemma}\label{lem5}
Let $H:\left[0,1\right]\times\mathbb{R}^\Lambda\to\mathbb{R}^\Lambda$ be defined as
\begin{equation*}
\left(s,\,\left(\phi_{ij}\right)_{i,j}\right)\mapsto\left(\psi_{ij}\right)_{i,j}\,,
\end{equation*}
\noindent where $\psi_{ij}=s^{\left|j-1\right|}\phi_{ij}$ and $0^0=1$. Suppose $\mathbb{R}^\Lambda$ has the product topology. Then $H$ is continuous.
\end{lemma}

\begin{proof}
Let $P_1:\left[0,1\right]\times\mathbb{R}^\Lambda\to\left[0,1\right]$ and $P_2:\left[0,1\right]\times\mathbb{R}^\Lambda\to\mathbb{R}^\Lambda$ be the projections onto the first and second coordinates respectively. For $(i,j)\in\Lambda$, let $P_{ij}:\mathbb{R}^\Lambda\to\mathbb{R}$ be the projection onto the $(i,j)^{\text{th}}$ coordinate. Note that the projections $P_1$, $P_2$ and $P_{ij}$, for $(i,j)\in\Lambda$, are continuous. For $(i,j)\in\Lambda$, let $Q_{ij}:\left[0,1\right]\times\mathbb{R}^\Lambda\to\left[0,1\right]\times\mathbb{R}$ be defined as
\begin{equation}
Q_{ij}(s,\phi)=\left(P_1(s,\phi),P_{ij}(P_2(s,\phi))\right)
\end{equation}
\noindent for $(s,\phi)\in\left[0,1\right]\times\mathbb{R}^\Lambda$. This map is continuous, since its component functions $P_1$ and $P_{ij}\circ P_2$ are continuous. For $j\in\mathbb{N}$, let $M_j:\left[0,1\right]\times\mathbb{R}\to\mathbb{R}$ be defined as $(s,t)\mapsto s^{\left|j-1\right|}t$. Note that this map is continuous. It can be checked that 
\begin{equation}
H(s,\phi)=\left(M_j(Q_{ij}(s,\phi))\right)_{i,j}
\end{equation}
\noindent for $(s,\phi)\in\left[0,1\right]\times\mathbb{R}^\Lambda$, where the tuple on the right has its $(i,j)^{\text{th}}$ coordinate $M_j(Q_{ij}(s,\phi))$ for $(i,j)\in\Lambda$. This shows that $H$ is continuous, since its component functions $M_j\circ Q_{ij}$, for $(i,j)\in\Lambda$, are continuous.
\end{proof}

\begin{lemma}\label{lem6}
Let $H$ be the map defined in Lemma \ref{lem5}. Then $H\left(\left[0,1\right]\times\mathcal{L}^n\right)\subseteq\mathcal{L}^n$ (for $n\in\mathbb{Z}^+$) and $H\left(\left[0,1\right]\times\mathcal{L}\right)\subseteq\mathcal{L}$.
\end{lemma}

\begin{proof}
Since $\mathcal{L}=\bigcup_{n\in\mathbb{Z}^+}\mathcal{L}^n$, it is enough to show that $H\left(\left[0,1\right]\times\mathcal{L}^n\right)\subseteq\mathcal{L}^n$ for $n\in\mathbb{Z}^+$. Suppose a positive integer $n$ be given. Let $(s,\phi)\in\left[0,1\right]\times\mathcal{L}^n$. Define $\psi:\mathbb{R}\to\mathbb{R}^\infty$ by
\begin{equation}
t\mapsto\sum_js^{\left|j-1\right|}\left(\phi_{ij}\right)_it^j
\end{equation}
\noindent where $\phi_{ij}$ is the coefficient of $t^j$ in the $i^{\text{th}}$ component of $\phi$ for $(i,j)\in\Lambda$. Since $\phi_{ij}=0$ for all but finitely many $(i,j)\in\Lambda$ and $\phi_{ij}=0$ whenever $i>n$, we have $\psi_{ij}=s^{\left|j-1\right|}\phi_{ij}=0$ for all but finitely many $(i,j)\in\Lambda$ and $\psi_{ij}=s^{\left|j-1\right|}\phi_{ij}=0$ whenever $i>n$. If $s\neq0$, then $\psi(u)-\psi(v)=\sum_{j\geq1}s^{\left|j-1\right|}\left(\phi_{ij}\right)_i(u^j-v^j)=s^{-1}\sum_{j\geq1}\left(\phi_{ij}\right)_i\left((su)^j-(sv)^j\right)=s^{-1}(\phi(su)-\phi(sv))\neq0$ for $u\neq v$ and $\psi'(t)=\sum_{j\geq1} j\left(\phi_{ij}\right)_i(st)^{j-1}=\phi'(st)\neq0$ for $t\in\mathbb{R}$ (since $\phi$ is a polynomial knot). Thus, if $s\neq0$, then $\psi$ is a polynomial knot. If $s=0$, then $\psi(t)=\left(\phi_{i1}t\right)_i$ for $t\in\mathbb{R}$, and hence $\psi$ is a linear knot (since $\phi_{i1}\neq0$ for some $i\leq n$). In either case, $\psi\in\mathcal{L}^n$. One can see that $H(s,\phi)=\psi$, and thus $H(s,\phi)\in\mathcal{L}^n$. This shows that $H\left(\left[0,1\right]\times\mathcal{L}^n\right)\subseteq\mathcal{L}^n$.
\end{proof}

By Lemma \ref{lem6}, the restriction of $H$ to $\left[0,1\right]\times\mathcal{L}$ has its image contained in $\mathcal{L}$. With this in mind, we have the following: 

\begin{lemma}\label{lem7}
Let $h:\left[0,1\right]\times\mathcal{L}\to\mathcal{L}$ be the restriction of the map $H$ defined in Lemma \ref{lem5}. Suppose $\mathcal{L}$ has any of the topologies $\mathcal{T}_b$, $\mathcal{T}_p$, $\mathcal{T}_r$ (for $r\geq1$) and $\mathcal{T}_\infty$. Then $h$ is continuous.
\end{lemma}

\begin{proof}
To prove the lemma, we consider the cases as follows:

(1) If $\mathcal{L}$ has the topology $\mathcal{T}_b$: Let $U=\mathcal{L}\bigcap\prod_{i,j} U_{ij}$ be a basic open set in $\left(\mathcal{L},\mathcal{T}_b\right)$. We show that $h^{-1}(U)$ is open in $\left[0,1\right]\times\mathcal{L}$. Let $(s,\phi)\in h^{-1}(U)$ be given. It is enough to show that $(s,\phi)\in V\times W\subseteq h^{-1}(U)$ for some open set $V$ in $\left[0,1\right]$ and an open set $W$ in $\mathcal{L}$. Since $h(s,\phi)=\left(s^{\left|j-1\right|}\phi_{ij}\right)_{i,j}$ and $h(s,\phi)\in U$, $s^{\left|j-1\right|}\phi_{ij}\in U_{ij}$ for $(i,j)\in\Lambda$. Let $\Lambda_\phi=\left\{(i,j)\in\Lambda:\phi_{ij}\neq0\right\}$. Note that $\Lambda_\phi$ is finite set. Choose an $\epsilon>0$ such that $\left(s^{\left|j-1\right|}\phi_{ij}-\epsilon,\hskip0.2mm s^{\left|j-1\right|}\phi_{ij}+\epsilon\right)\subseteq U_{ij}$ for all $(i,j)\in\Lambda_\phi$. Since $s^{\left|j-1\right|}\phi_{ij}=0$ for $(i,j)\in\Lambda\setminus\Lambda_\phi$, $0\in U_{ij}$ for all $(i,j)\in\Lambda\setminus\Lambda_\phi$. For $(i,j)\in\Lambda\setminus\Lambda_\phi$, let $\epsilon_{ij}>0$ be such that $\left(-\epsilon_{ij},\epsilon_{ij}\right)\subseteq U_{ij}$. Let $m=\max\left\{j\in\mathbb{N}:(i,j)\in\Lambda_\phi\;\text{for some}\;i\in\mathbb{Z}^+\right\}$ and $M=\max\left\{\left|\phi_{ij}\right|:\,(i,j)\in\Lambda_\phi\right\}$. We can find a $\delta>0$ such that
\begin{equation}\label{eq1}
\left|s^{\left|j-1\right|}-t^{\left|j-1\right|}\right|<\frac{\epsilon}{3M}\qquad\text{whenever}\;\left|s-t\right|<\delta\;\text{and}\;j\leq m
\end{equation}
\noindent Let $V=\left(s-\delta, s+\delta\right)\cap\left[0,1\right]$ and $W=\mathcal{L}\bigcap\prod_{i,j} W_{ij}$, where $W_{ij}=\left(\phi_{ij}-\epsilon/3,\phi_{ij}+\epsilon/3\right)$ for $(i,j)\in\Lambda_\phi$ and $W_{ij}=\left(-\epsilon_{ij},\epsilon_{ij}\right)$ for $(i,j)\in\Lambda\setminus\Lambda_\phi$. Note that $(s,\phi)\in V\times W$. Let $(t,\psi)\in V\times W$ be given. By Expression (\ref{eq1}),
\begin{equation} 
\left|s^{\left|j-1\right|}\phi_{ij}-t^{\left|j-1\right|}\psi_{ij}\right|\leq\left|s^{\left|j-1\right|}-t^{\left|j-1\right|}\right|\left|\phi_{ij}\right|+t^{\left|j-1\right|}\left|\phi_{ij}-\psi_{ij}\right|<\epsilon
\end{equation}
\noindent for $(i,j)\in\Lambda_\phi$, since $\left|s-t\right|<\delta$ and $j\leq m$. Thus, $t^{\left|j-1\right|}\psi_{ij}\in U_{ij}$ for $(i,j)\in\Lambda_\phi$. Also, $t^{\left|j-1\right|}\psi_{ij}\in\left(-\epsilon_{ij},\epsilon_{ij}\right)\subseteq U_{ij}$ for $(i,j)\in\Lambda\setminus\Lambda_\phi$. Therefore, $h(t,\psi)=\left(t^{\left|j-1\right|}\psi_{ij}\right)_{i,j}\in U$, i.e. $(t,\psi)\in h^{-1}(U)$. This proves that $V\times W\subseteq h^{-1}(U)$.

(2) If $\mathcal{L}$ has the topology $\mathcal{T}_p$: Since $\mathcal{T}_p$ is the subspace topology inherited by the product topology of $\mathbb{R}^\Lambda$, the map $h$ (being the restriction of $H$) is obviously continuous by Lemma \ref{lem5}.

(3) If $\mathcal{L}$ has the topology $\mathcal{T}_r$: Suppose $(s,\phi)\in\left[0,1\right]\times\mathcal{L}$. We show that $h$ is continuous at $(s,\phi)$. Let $\epsilon>0$ be given. Let $\delta=\epsilon/3$, $m=\max\left\{j\in\mathbb{N}:\phi_{ij}\neq0\;\text{for some}\;i\in\mathbb{Z}^+\right\}$ and $R=\left(\sum_{i\geq1}\sum_{j\leq m}\left|\phi_{ij}\right|^r\right)^{1/r}$. Choose a $\delta'>0$ such that 
\begin{equation}\label{eq2}
\left|s^{\left|j-1\right|}-t^{\left|j-1\right|}\right|<\frac{\delta}{R}\qquad\text{whenever}\;\left|s-t\right|<\delta'\;\text{and}\;j\leq m
\end{equation}
\noindent Since $\phi_{ij}=0$ for $i\geq1$ and $j>m$, using Expression (\ref{eq2}), we have
{\small\begin{equation}
d_r(h(s,\phi),h(t,\phi))\leq\left(\sum_{i\geq1}\sum_{j\leq m}\left|s^{\left|j-1\right|}-t^{\left|j-1\right|}\right|^r\left|\phi_{ij}\right|^r\right)^{1/r}\leq\frac{\delta}{R}\left(\sum_{i\geq1}\sum_{j\leq m}\left|\phi_{ij}\right|^r\right)^{1/r}\leq\frac{\epsilon}{3}
\end{equation}}
\noindent whenever $\left|s-t\right|<\delta'$. Also,
{\small\begin{equation}
d_r(h(t,\phi),h(t,\psi))\leq\left(\sum_{i,j}t^{\left|j-1\right|r}\left|\phi_{ij}-\psi_{ij}\right|^r\right)^{1/r}\leq\left(\sum_{i,j}\left|\phi_{ij}-\psi_{ij}\right|^r\right)^{1/r}\leq\frac{\epsilon}{3}
\end{equation}}
\noindent whenever $d_r(\phi,\psi)<\delta$. Therefore,
\begin{equation}
d_r(h(s,\phi),h(t,\psi))\leq d_r(h(s,\phi),h(t,\phi))+d_r(h(t,\phi),h(t,\psi))<\epsilon
\end{equation}
\noindent whenever $\left|s-t\right|<\delta'$ and $d_r(\phi,\psi)<\delta$. This proves that $h$ is continuous at $(s,\phi)$.

(4) If $\mathcal{L}$ has the topology $\mathcal{T}_\infty$: By slight modifications in the previous case, it can be shown that $h$ is continuous in the present case.
\end{proof}

The restriction of $H$ to $\left[0,1\right]\times\mathcal{L}^n$ has its image contained in $\mathcal{L}^n$ (see Lemma \ref{lem6}). Since $\mathcal{T}_b^n$, $\mathcal{T}_p^n$, $\mathcal{T}_r^n$ (for $r\geq1$) and $\mathcal{T}_\infty^n$ are the subspace topologies that inherit from the topologies $\mathcal{T}_b$, $\mathcal{T}_p$, $\mathcal{T}_r$ and $\mathcal{T}_\infty$ respectively, we have an immediate consequence of Lemma \ref{lem7} as follows:

\begin{lemma}\label{lem8}
Let $h_n:\left[0,1\right]\times\mathcal{L}^n\to\mathcal{L}^n$ be the restriction of the map $H$ defined in Lemma \ref{lem5}. Suppose $\mathcal{L}^n$ has any of the topologies $\mathcal{T}_b^n$, $\mathcal{T}_p^n$, $\mathcal{T}_r^n$ (for $r\geq1$) and $\mathcal{T}_\infty^n$. Then $h_n$ is continuous.
\end{lemma}

\begin{lemma}\label{lem9}
The space $\mathcal{E}^n$ having any of the topologies $\mathcal{S}_b^n$, $\mathcal{S}_p^n$, $\mathcal{S}_r^n$ (for $r\geq1$) and $\mathcal{S}_\infty^n$ is homeomorphic to $\mathbb{R}^n\setminus\{0\}$.
\end{lemma}

\begin{proof}
Let $\alpha:\mathcal{E}^n\to\mathbb{R}^n\setminus\{0\}$ be a map defined as $(x_1,x_2,x_3,\ldots)\mapsto(x_1,x_2,\ldots,x_n)$ and suppose $\beta:\mathbb{R}^n\setminus\{0\}\to\mathcal{E}^n$ be defined as $(y_1,y_2,\ldots,y_n)\mapsto(y_1,y_2,\ldots,y_n,0,0,\ldots)$. Note that the compositions $\alpha\circ\beta$ and $\beta\circ\alpha$ are the identity maps of $\mathbb{R}^n\setminus\{0\}$ and $\mathcal{E}^n$ respectively. Thus, $\alpha$ is a bijective map and $\beta$ is its inverse. Now, it is enough to show that $\alpha$ and $\beta$ are continuous. To do this, we consider the cases as follows:

(1) If $\mathcal{E}^n$ has the topology $\mathcal{S}_b^n$: The sets of the type $\mathcal{E}^n\bigcap\prod_i U_i$, where $U_i$ is open in $\mathbb{R}$ for $i\in\mathbb{Z}^+$, form a basis for the topology $\mathcal{S}_b^n$. Let $U=\mathcal{E}^n\bigcap\prod_i U_i$ be an open set in this basis. If $0\notin U_i$ for some $i>n$, then $U$ is an empty set and hence so is the set $\beta^{-1}(U)$. If $0\in U_i$ for all $i>n$, then $\beta^{-1}(U)=U_1\times U_2\times\cdots\times U_n\setminus\{(0,0,\ldots,0)\}$. In either case, $\beta^{-1}(U)$ is open in $\mathbb{R}^n\setminus\{0\}$. This shows that $\beta$ is continuous. It can be seen that the sets of the type $V_1\times V_2\times\cdots\times V_n\setminus\{(0,0,\ldots,0)\}$, where $V_i$ is open in $\mathbb{R}$ for $i=1,2,\ldots,n$, form a basis for the standard topology of $\mathbb{R}^n\setminus\{0\}$. Suppose $V=V_1\times V_2\times\cdots\times V_n\setminus\{(0,0,\ldots,0)\}$ be an open set in this basis. Then $\alpha^{-1}(V)=\mathcal{E}^n\bigcap\prod_i W_i$, where $W_i=V_i$ for $i\leq n$ and $W_i=\mathbb{R}$ for $i>n$. Thus, $\alpha^{-1}(V)$ is open in $(\mathcal{E}^n,\mathcal{S}_b^n)$. This shows that $\alpha$ is continuous.

(2) If $\mathcal{E}^n$ has the topology $\mathcal{S}_p^n$: By slight modifications in the previous case, it can shown that $\alpha$ and $\beta$ are continuous in the present case.

(3) If $\mathcal{E}^n$ has the topology $\mathcal{S}_r^n$: Open balls of the type $C_r^n(x,\delta)=\left\{z\in\mathcal{E}^n:\rho_r(x,z)<\delta\right\}$, for $x\in\mathcal{E}^n$ and $\delta>0$, form a basis for the topology $\mathcal{S}_r^n$, and the sets of the type 
\begin{equation*}
U_r^n(y,\delta)=\left\{w\in\mathbb{R}^n\setminus\{0\}:\left(\sum_{i=1}^n\left|y_i-w_i\right|^r\right)^{1/r}<\delta\right\}
\end{equation*}
\noindent for $y\in\mathbb{R}^n\setminus\{0\}$ and $\delta>0$, form a basis for the usual topology of $\mathbb{R}^n\setminus\{0\}$. It can be checked that $\beta^{-1}(C_r^n(x,\delta))=U_r^n(\alpha(x),\delta)$ and $\alpha^{-1}(U_r^n(y,\delta))=C_r^n(\beta(y),\delta)$ for $x\in\mathcal{E}^n$, $y\in\mathbb{R}^n\setminus\{0\}$ and $\delta>0$. Thus, $\alpha$ and $\beta$ are continuous.

(4) If $\mathcal{E}^n$ has the topology $\mathcal{S}_\infty^n$: By slight modifications in the previous case, it can be shown that $\alpha$ and $\beta$ are continuous in the present case.
\end{proof}

\begin{proof}[Proof of Theorem \ref{th1}]
The composition $f_n\circ g_n$ is the identity map of $\mathcal{E}^n$, where $f_n$ and $g_n$ be the maps defined in Lemma \ref{lem4}. Also, it can be checked that $h_n(0,\phi)=g_n(f_n(\phi))$ and $h_n(1,\phi)=\phi$ for $\phi\in\mathcal{L}^n$, where $h_n$ is the map defined in Lemma \ref{lem8}. Thus, the map $g_n\circ f_n$ is homotopy equivalent to the identity map of $\mathcal{L}^n$, and hence the spaces $\mathcal{E}^n$ and $\mathcal{L}^n$, with the topologies as in Lemma \ref{lem4}, have the same homotopy type. Since $\mathbb{R}^n\setminus\{0\}$ and $S^{n-1}$ have the same homotopy type, the result follows by Lemma \ref{lem9}.
\end{proof}

\begin{lemma}\label{lem10}
The space $\mathcal{E}$ having any of the topologies $\mathcal{S}_b$, $\mathcal{S}_p$, $\mathcal{S}_r$ (for $r\geq1$) and $\mathcal{S}_\infty$ is contractible.
\end{lemma}

\begin{proof}
Let $S:\left[0,1\right]\times\mathcal{E}\to\mathcal{E}$ and $T:\left[0,1\right]\times\mathcal{E}\to\mathcal{E}$ be maps defined as $S(s,\left(x_i\right)_i)=\left((1-s)x_i+sx_{i-1}\right)_i$ and $T(s,\left(x_i\right)_i)=\left((1-s)x_{i-1}+sa_i\right)_i$ for $(s,\left(x_i\right)_i)\in\left[0,1\right]\times\mathcal{E}$, where $x_0=0$, $a_1=1$ and $a_i=0$ for $i\geq2$. Note that $S(0,x)=\text{Id}_\mathcal{E}(x)$, $S(1,x)=T(0,x)=S_1(x)$ and $T(1,x)=C_a(x)$ for $x\in\mathcal{E}$, where $\text{Id}_\mathcal{E}$ is the identity map of $\mathcal{E}$, $S_1$ is the shift map $(x_1,x_2,x_3,\ldots)\mapsto(0,x_1,x_2,\ldots)$ and $C_a$ is the constant map $(x_1,x_2,x_3,\ldots)\mapsto(1,0,0,\ldots)$. If we can show that $S$ and $T$ are continuous, then the identity map of $\mathcal{E}$ will be null-homotopic and hence $\mathcal{E}$ will be contractible. To show that $S$ and $T$ are continuous, we consider the cases as follows:

(1) If $\mathcal{E}$ has the topology $\mathcal{S}_b$:

(a) To prove $S$ is continuous: Let $U=\mathcal{E}\bigcap\prod_i U_i$ be a basic open set in $(\mathcal{E},\mathcal{S}_b)$. We show that $S^{-1}(U)$ is open in $\left[0,1\right]\times\mathcal{E}$. Let $(s,\left(x_i\right)_i)\in S^{-1}(U)$ be given. It is enough to show that $(s,\left(x_i\right)_i)\in V\times W\subseteq S^{-1}(U)$ for some open set $V$ in $\left[0,1\right]$ and an open set $W$ in $\mathcal{E}$. Let $m\geq2$ be such that $x_i=0$ for all $i\geq m$ and let $M=\max\left\{\left|x_i\right|:i<m\right\}$. Since $S(s,\left(x_i\right)_i)\in U$, $(1-s)x_i+sx_{i-1}\in U_i$ for $i\in\mathbb{Z}^+$. Choose an $\epsilon>0$ such that $\left((1-s)x_i+sx_{i-1}-\epsilon,\hskip0.2mm(1-s)x_i+sx_{i-1}+\epsilon\right)\subseteq U_i$ for all $i\leq m$. Since $x_i=x_{i-1}=0$ for $i>m$, $0\in U_i$ for all $i>m$. For $i>m$, choose an $\epsilon_i>0$ such that $(-\epsilon_i,\epsilon_i)\subseteq U_i$. Let $\delta=\min\{\frac{\epsilon}{4},\frac{\epsilon_{m+1}}{3}\}$ and $\delta_i=\min\{\frac{\epsilon_i}{3},\frac{\epsilon_{i+1}}{3}\}$ for $i>m$. Take $V=\left(s-\delta/M,\hskip0.2mm s+\delta/M\right)\cap\left[0,1\right]$ and $W=\mathcal{E}\bigcap\prod_i W_i$, where $W_i=(x_i-\delta,x_i+\delta)$ for $i\leq m$ and $W_i=(-\delta_i,\delta_i)$ for $i>m$. Note that $(s,\left(x_i\right)_i)\in V\times W$. Let $(t,\left(y_i\right)_i)\in V\times W$ be given. Since $\left|s-t\right|<\delta/M$ and $\left|x_i-y_i\right|<\delta$ for $i\leq m$, we have
\begin{align}
\left|(1-s)x_i+sx_{i-1}-(1-t)y_i-ty_{i-1}\right|&\leq\left|(1-s)-(1-t)\right|\left|x_i\right|+\left|s-t\right|\left|x_{i-1}\right|\nonumber\\
&\hskip4.1mm+(1-t)\left|x_i-y_i\right|+t\left|x_{i-1}-y_{i-1}\right|\nonumber\\
&\leq\frac{\delta}{M}M+\frac{\delta}{M}M+(1-t)\delta+t\delta\nonumber\\
&<\epsilon
\end{align}
\noindent for $i\leq m$. Thus, $(1-t)y_i+ty_{i-1}\in U_i$ for $i\leq m$. Since $x_m=0$, $\left|x_m-y_m\right|<\delta$ and $\left|y_i\right|<\delta_i$ for $i>m$, we have $\left|(1-t)y_{m+1}+ty_m\right|\leq\left|y_{m+1}\right|+\left|x_m-y_m\right|\leq\delta_{m+1}+\delta<\epsilon_{m+1}$ and $\left|(1-t)y_i+ty_{i-1}\right|\leq\left|y_i\right|+\left|y_{i-1}\right|\leq\delta_i+\delta_{i-1}<\epsilon_i$ for $i\geq m+2$. In other words $(1-t)y_i+ty_{i-1}\in U_i$ for $i>m$. Therefore, $S(t,\left(y_i\right)_i)=\left((1-t)y_i+ty_{i-1}\right)_i\in U$, i.e. $(t,\left(y_i\right)_i)\in S^{-1}(U)$. This shows that $V\times W\subseteq S^{-1}(U)$.

(b) To prove $T$ is continuous: Let $U=\mathcal{E}\bigcap\prod_i U_i$ be a basic open set in $(\mathcal{E},\mathcal{S}_b)$. We show that $T^{-1}(U)$ is open in $\left[0,1\right]\times\mathcal{E}$. Let $(s,\left(x_i\right)_i)\in T^{-1}(U)$ be given. It is enough to show that $(s,\left(x_i\right)_i)\in V\times W\subseteq T^{-1}(U)$ for an open set $V$ in $\left[0,1\right]$ and an open set $W$ in $\mathcal{E}$. Let $m\geq2$ be such that $x_i=0$ for all $i\geq m$ and let $M=\max\left\{\left|x_i\right|:i<m\right\}$. Since $T(s,\left(x_i\right)_i)\in U$, $(1-s)x_{i-1}+sa_i\in U_i$ for $i\in\mathbb{Z}^+$. Choose an $\epsilon>0$ such that $\left((1-s)x_{i-1}+sa_i-\epsilon,\hskip0.2mm(1-s)x_{i-1}+sa_i+\epsilon\right)\subseteq U_i$ for all $i\leq m$. Since $x_{i-1}=a_i=0$ for $i>m$, $0\in U_i$ for all $i>m$. For $i>m$, choose an $\epsilon_i>0$ such that $(-\epsilon_i,\epsilon_i)\subseteq U_i$. Let $\delta=\min\{\frac{\epsilon}{3},\frac{\epsilon_{m+1}}{2}\}$ and $\delta_i=\frac{\epsilon_{i+1}}{2}$ for $i>m$. Take $V=\left(s-\frac{\delta}{M+1},\hskip0.2mm s+\frac{\delta}{M+1}\right)\cap\left[0,1\right]$ and $W=\mathcal{E}\bigcap\prod_i W_i$\,, where $W_i=(x_i-\delta,x_i+\delta)$ for $i\leq m$ and $W_i=(-\delta_i,\delta_i)$ for $i>m$. Note that $(s,\left(x_i\right)_i)\in V\times W$. Let $(t,\left(y_i\right)_i)\in V\times W$ be given. Since $\left|s-t\right|<\frac{\delta}{M+1}$ and $\left|x_i-y_i\right|<\delta$ for $i\leq m$, we have
\begin{align}
\left|(1-s)x_{i-1}+sa_i-(1-t)y_{i-1}-ta_i\right|&\leq\left|(1-s)-(1-t)\right|\left|x_{i-1}\right|+(1-t)\left|x_{i-1}-y_{i-1}\right|\nonumber\\
&\hskip4mm+\left|s-t\right|\left|a_i\right|\nonumber\\
&\leq\frac{\delta}{M+1}M+\delta+\frac{\delta}{M+1}\nonumber\\
&<\epsilon
\end{align}
\noindent for $i\leq m$. Thus, $(1-t)y_{i-1}+ta_i\in U_i$ for $i\leq m$. Since $x_m=0$, $\left|x_m-y_m\right|<\delta$, $a_i=0$ and $\left|y_i\right|<\delta_i$ for $i>m$, we have $\left|(1-t)y_m+ta_{m+1}\right|\leq\left|x_m-y_m\right|<\epsilon_{m+1}$ and $\left|(1-t)y_{i-1}+ta_i\right|\leq\delta_{i-1}<\epsilon_i$ for $i\geq m+2$. In other words $(1-t)y_{i-1}+ta_i\in(-\epsilon_i,\epsilon_i)\subseteq U_i$ for $i>m$. Therefore, $T(t,\left(y_i\right)_i)=\left((1-t)y_{i-1}+ta_i\right)_i\in U$, i.e. $(t,\left(y_i\right)_i)\in T^{-1}(U)$. This proves that $V\times W\subseteq T^{-1}(U)$.

(2) If $\mathcal{E}$ has the topology $\mathcal{S}_p$: 

(a) To prove $S$ is continuous: Let $U=\mathcal{E}\bigcap\prod_i U_i$ be a basic open set in $(\mathcal{E},\mathcal{S}_p)$. We show that $S^{-1}(U)$ is open in $\left[0,1\right]\times\mathcal{E}$. Since $U$ is also open in $(\mathcal{E},\mathcal{S}_b)$, $S^{-1}(U)$ is open in $\left[0,1\right]\times\mathcal{E}$ with respect to the topology $\mathcal{S}_b$ on $\mathcal{E}$. Thus, we have $S^{-1}(U)=\bigcup_jV_j\times W_j$, where $j$ runs over some indexing set $J$, and for $j\in J$, $V_j$ is a basic open set in $\left[0,1\right]$ and $W_j=\mathcal{E}\bigcap\prod_iW_{ji}$ is a basic open set in $(\mathcal{E},\mathcal{S}_b)$. We assume that $V_j $ and $W_j$ are nonempty for all $j\in J$. The set $W_{ji}$ is nonempty for all $(i,j)\in\mathbb{Z}^+\times J$, and for $j\in J$, $0\in W_{ji}$ for all but finitely many $i\in\mathbb{Z}^+$. Since $U_i=\mathbb{R}$ for all but finitely many $i\in\mathbb{Z}^+$, we can find an integer $m$ such that $U_i=\mathbb{R}$ for all $i>m$. For $j\in J$, let $W'_j=\mathcal{E}\bigcap\prod_iW'_{ji}$, where $W'_{ji}=W_{ji}$ for $i\leq m$ and $W'_{ji}=\mathbb{R}$ for $i>m$. Note that $W'_j$ is a basic open set in $(\mathcal{E},\mathcal{S}_p)$ for each $j\in J$. It is enough to prove that $S^{-1}(U)=\bigcup_jV_j\times W'_j$. We note that $S^{-1}(U)\subseteq\bigcup_jV_j\times W'_j$, since $W_j\subseteq W'_j$ for $j\in J$. Suppose $(s,\left(x_i\right)_i)\in\bigcup_jV_j\times W'_j$. Then $(s,\left(x_i\right)_i)\in V_j\times W'_j$ for some $j\in J$ and hence $x_i\in W'_{ji}=W_{ji}$ for $i\leq m$. Let $y_i=x_i$ for $i\leq m$. Since $W_{ji}$ is nonempty for $i>m$ and $0\in W_{ji}$ for all but finitely many $i>m$, we can choose $y_i\in W_{ji}$ for each $i>m$ such that at least one of the $y_i$ is nonzero and $y_i=0$ for all but finitely many $i>m$. It is clear that $(s,\left(y_i\right)_i)\in V_j\times W_j\subseteq S^{-1}(U)$. In other words $S(s,\left(y_i\right)_i)\in U$. Thus, $(1-s)x_i+sx_{i-1}=(1-s)y_i+sy_{i-1}\in U_i$ for $i\leq m$, since $y_i=x_i$ for $i\leq m$ and $S(s,\left(y_i\right)_i)=\left((1-s)y_i+sy_{i-1}\right)_i$. Note that $(1-s)x_i+sx_{i-1}\in U_i$ for $i>m$ (since $U_i=\mathbb{R}$ for $i>m$). Therefore, $S(s,\left(x_i\right)_i)\in U$, i.e. $(s,\left(x_i\right)_i)\in S^{-1}(U)$. This shows that $\bigcup_jV_j\times W'_j\subseteq S^{-1}(U)$.

(b) The continuity of $T$ can be proved in a similar way as in the case of $S$.

(3) If $\mathcal{E}$ has the topology $\mathcal{S}_r$:

(a) To prove $S$ is continuous: Let $(s,\left(x_i\right)_i)\in\left[0,1\right]\times\mathcal{E}$. We show that $S$ is continuous at $(s,\left(x_i\right)_i)$. Suppose an $\epsilon>0$ be given. Take $\delta=\frac{\epsilon}{3}$ and $\delta'=\frac{\epsilon}{6R}$\,, where $R=\left(\sum_i\left|x_i\right|^r\right)^{1/r}$. We have the following:
\begin{align}
\rho_r(S(s,\left(x_i\right)_i),S(t,\left(y_i\right)_i))&\leq\left(\sum_i\left|(1-s)x_i+sx_{i-1}-(1-t)x_i-tx_{i-1}\right|^r\right)^{1/r}\nonumber\\
&\hskip4.5mm+\left(\sum_i\left|(1-t)x_i+tx_{i-1}-(1-t)y_i-ty_{i-1}\right|^r\right)^{1/r}\nonumber\\
&\leq2\left|s-t\right|\left(\sum_i\left|x_i\right|^r\right)^{1/r}+(1-t+t)\left(\sum_i\left|x_i-y_i\right|^r\right)^{1/r}\nonumber\\
&\leq2\hskip0.2mm\delta'R+\delta\nonumber\\
&<\epsilon
\end{align}
\noindent whenever $\left|s-t\right|<\delta'$ and $\rho_r(\left(x_i\right)_i,\left(y_i\right)_i)<\delta$. This proves that $S$ is continuous at $(s,\left(x_i\right)_i)$.

(b) To prove $T$ is continuous: Let $(s,\left(x_i\right)_i)\in\left[0,1\right]\times\mathcal{E}$. We show that $T$ is continuous at $(s,\left(x_i\right)_i)$. Suppose an $\epsilon>0$ be given. Take $\delta=\frac{\epsilon}{3}$ and $\delta'=\frac{\epsilon}{3+3R}$\,, where $R=\left(\sum_i\left|x_i\right|^r\right)^{1/r}$. We have the following:
\begin{align}
\rho_r(T(s,\left(x_i\right)_i),T(t,\left(y_i\right)_i))&\leq\left(\sum_i\left|(1-s)x_{i-1}+sa_i-(1-t)y_{i-1}-ta_i\right|^r\right)^{1/r}\nonumber\\
&\leq\left(\sum_i\left|s-t\right|^r\left|a_i\right|^r\right)^{1/r}+\left(\sum_i\left|(1-s)-(1-t)\right|^r\left|x_{i-1}\right|^r\right)^{1/r}\nonumber\\
&\hskip4.5mm+\left(\sum_i(1-t)^r\left|x_{i-1}-y_{i-1}\right|^r\right)^{1/r}\nonumber\\
&\leq\delta'+\delta'R+\delta\nonumber\\
&<\epsilon
\end{align}
\noindent whenever $\left|s-t\right|<\delta'$ and $\rho_r(\left(x_i\right)_i,\left(y_i\right)_i)<\delta$. This proves that $T$ is continuous at $(s,\left(x_i\right)_i)$.

(4) If $\mathcal{E}$ has the topology $\mathcal{S}_\infty$: By sight modifications in the previous case, it can shown that $S$ and $T$ are continuous in the present case. 
\end{proof}

\begin{proof}[Proof of Theorem \ref{th2}]
The composition $f\circ g$ is the identity map of $\mathcal{E}$, where $f$ and $g$ be the maps defined in Lemma \ref{lem3}. Also, it can be seen that $h(0,\phi)=g(f(\phi))$ and $h(1,\phi)=\phi$ for $\phi\in\mathcal{L}$, where $h$ is the map defined in Lemma \ref{lem7}. Thus, the map $g\circ f$ is homotopy equivalent to the identity map of $\mathcal{L}$, and hence the spaces $\mathcal{E}$ and $\mathcal{L}$, with the topologies as in Lemma \ref{lem3}, have the same homotopy type. Now, the result follows by Lemma \ref{lem10}.
\end{proof}

\section*{Acknowledgments}

The author is thankful to the Harish-Chandra Research Institute, Prayagraj, India for the postdoctoral fellowship during this research work.

\end{document}